\documentclass[wasysym,reqno]{amsart}
\usepackage{graphicx}
\usepackage{amsmath}
\usepackage{amssymb}
\usepackage{amsfonts}
\usepackage{mathrsfs}
\usepackage{cite}

\newtheorem{thm}{Theorem}[section]
\newtheorem{cor}[thm]{Corollary}
\newtheorem{lem}[thm]{Lemma}

\theoremstyle{definition}
\newtheorem{defn}[thm]{Definition}
\theoremstyle{remark}

\numberwithin{equation}{section}

\begin{document}

\title[Derivations on \text{FCIN} algebras]{Derivations on FCIN algebras}

\author{Asia Majeed \& Cenap \"{o}zel}
\thanks{}

\address{Department of Mathematics, CIIT, Islamabad, Pakistan \& AIBU Izzet Baysal
University Department of Mathematics Bolu / Turkey}

\email{asiamajeed@hotmail.com;cenap.ozel@gmail.com;cenap@ibu.edu.tr}



\thanks{}

\begin{abstract}
Let $\mathcal{L}$ be an algebra generated by the commuting
independent nests, $\mathcal{M}$ is an ultra-weakly closed
subalgebra of $\mathbf{B(H)}$ which contains $alg\mathcal{L}$ and
$\phi$ is a norm continuous linear mapping from $alg\mathcal{L}$
into $\mathcal{M}$. In this paper we will show that a norm
continuous linear derivable mapping at zero point from
$Alg\mathcal{L}$ to $\mathcal{M}$ is a derivation.
\end{abstract}

\maketitle
\commby{}
\section{Introduction.}

\begin{defn}
Let $\mathcal{A}$ be a subalgebra of $\mathbf{B(H)}$, let $\phi$ be
a linear mapping from $\mathcal{A}$ to $\mathbf{B(H)}$. We say that
$\phi$ is a \textit{derivation} if $\phi(AB)=\phi(A)B+A\phi(B)$ for
any $A,B$ $\in \mathcal{A}$.

We say that $\phi$ is a \textit{derivable mapping at the zero point}
if $\phi(AB) = \phi(A)B+A\phi(B)$ for any $A, B \in \mathcal{A}$
with $AB =0$.
\end{defn}
Several authors have studied linear mappings on operator algebras
are derivations. In \cite{baer1} Jing and Liu showed that every
derivable mapping $\phi$ at $0$ with $\phi(I)=0$ on nest algebras is
an inner derivation. In \cite{tropic,J} Zhu and Xiong proved that
every norm continuous generalized derivable mapping at $0$ on a
finite \text{CSL} algebra is a generalized derivation, and every
strongly operator topology continuous derivable mapping at the unit
operator $I$ in nest algebras is a derivation. It is natural and
interesting to ask whether or not a linear mapping is a derivation
if it is derivable only at one given point. An and Hou \cite{baer2}
investigated derivable mapping at $0, P$, and $I$ on triangular
rings, where $P$ is a fixed non-trivial idempotent. In \cite{S} Zhao
and Zhu characterized Jordan derivable mappings at $0$ and $I$ on
triangular algebras.

Now we will give some required definitions.
\begin{defn}
Let $\mathcal{L}$ be a lattice on a Hilbert space $\mathbf H$. If
$\mathcal{L}$ is generated by finitely many commuting independent
nests, it will be called by an $alg\mathcal{L}$  \text{FCIN}
algebra.
\end{defn}
Let $\mathcal{L}$ be a subspace lattice. For each projection $E \in
\mathcal{L}$, let
$$E_{-}= \bigvee{\{F:F \in \mathcal{L}, F\ngeq E}\} \quad
\text{and}\quad E_{*}= \bigvee{\{F_{-}:F \in \mathcal{L}, F\nleq
E}\}.$$

\begin{defn} A subspace lattice $\mathcal{L}$ is called
\bf{completely distributive} if $E_{*}= E, \forall E \in
\mathcal{L}$.
\end{defn}
\begin{defn}
If $\mathcal{L}$ is completely distributive and commutative, we will
call an $Alg\mathcal{L}$  \text{CDCSL} algebra.
\end{defn}
Throughout we consider $x, y$ be vectors in $\mathbf H$, we use
 notation $x\otimes y$ for rank one operators defined by $(x\otimes y )z = (z,x)y$ for all $z \in \mathbf H$.
Let $\mathcal{R}_{L}$ be the spanning space of rank one operators in
$Alg\mathcal{L}$. Laurie and Longstaff \cite{baer3} proved the
following result.
\begin{thm}
A commutative subspace lattice $\mathcal{L}$ is completely
distributive if and only if $\mathcal{R}_{L}$ is ultra-weakly dense
in $Alg\mathcal{L}$.
\end{thm}
\begin{defn}
Let $\mathcal{L}$ be a \text{CSL}. Then the von Neumann algebra
$(Alg\mathcal{L})\cap (Alg\mathcal{L})^{*}$ is called \bf{diagonal}
of $Alg\mathcal{L}$ and denoted by $D(\mathcal{L})$.
\end{defn}
Assume that $\mathcal{L}$ is generated by the commuting independent
$${\{E(Alg\mathcal{L})E^{\bot}: E\in \mathcal{L}}\}.$$
It is clear that $\mathcal{R}_{L}$ is a norm closed ideal of the
\text{CSL} algebra $Alg\mathcal{L}$.

Assume that $\mathcal{L}$ is generated by the commuting independent
nests $\mathcal{L}_{1}, \mathcal{L}_{2},...., \mathcal{L}_{n}$, then
$\mathcal{M}$ is an ultra-weakly closed subalgebra of
$\mathbf{B(H)}$ which contains $alg\mathcal{L}$, and $\phi$ is a
norm continuous linear mapping from $Alg\mathcal{L}$ into
$\mathcal{M}$.

\section{The Main Result.}

To prove the main result of this paper we require the following
Lemma from \cite{baer4}.
\begin{lem}
Let $\mathcal{L}$ be an arbitrary \text{CSL} on the complex
separable Hilbert space $\mathbf H$, and $\mathcal{M}$ be an
ultra-weakly closed subalgebra of $\mathbf{B(H)}$ which contains
$Alg\mathcal{L}$.\\ If $\phi: Alg\mathcal{L} \rightarrow
\mathcal{M}$ is a norm continuous linear mapping, then
$\phi(XAY)=\phi(XA)Y+X\phi(AY)-X\phi(A)Y$ for all $A$ in
$Alg\mathcal{L}$ and all $X, Y$ in
$\mathcal{D}(\mathcal{L})+\mathcal{R}(\mathcal{L})$.
\end{lem}

\begin{lem}\label{abbas}
Let $A \in Alg\mathcal{L}$ and $B\in
\mathcal{D}(\mathcal{L})+\mathcal{R}(\mathcal{L})$. If $AB
\in\mathcal{D}(\mathcal{L})+\mathcal{R}(\mathcal{L})$, then
$\phi(AB)=\phi(A)B+A\phi(B)$.
\end{lem}
\begin{cor}
 $\phi(XY)=\phi(X)Y+X\phi(Y)$ for all~ $X, Y \in$
$\mathcal{D}(\mathcal{L})+\mathcal{R}(\mathcal{L})$.
\end{cor}
So we are ready to prove the main result of this work.
\begin{thm}
Let $\mathcal {L}$ be a commutative subspace lattice generated by
finitely many independent nests, and $\mathcal {M}$ be any
ultra-weakly closed subalgebra of $\mathbf{B(H)}$ on $\mathbf H$,
which contains $Alg\mathcal{L}$. Let $\phi$ be a norm continuous
linear derivable mapping at the zero point from $Alg\mathcal{L}$ to
$\mathcal{M}$. Then $\phi$ is a derivation.
\end{thm}

\begin{proof}
Let $\phi: Alg\mathcal{L}\rightarrow\mathcal{M}$ be a norm
continuous derivable linear mapping. Then we just need to prove that
$\phi$ is a derivation. Let $\Omega=\{i:I_{-}=I~
\in~\mathcal{L}_{i}\}$. Then we have the following cases.
\par
When $\Omega=\Phi$, for each i=1, 2,....., n,  let $Q_{i}=I_{-}$ be
the projection in $\mathcal{L}_{i}$ and
N=$\prod_{i=1}^{n}Q_{i}^{\bot}$, then $\mathbf{B(H)}N
\subset\mathcal{D}(\mathcal{L})+\mathcal{R}(\mathcal{L})$. Let $A, B
\in alg\mathcal{L}$, it follows from Lemma \ref{abbas} that
\begin{eqnarray}
\phi(ABTN)=\phi(AB)TN+AB\phi(TN)\nonumber
\end{eqnarray}
for all $T \in B(H)$. On the other hand, we have from Lemma
\ref{abbas} again
\begin{eqnarray}
\phi(ABTN)&=&\phi(A)BTN+A\phi(BTN)\nonumber\\
&=&\phi(A)BTN+A[\phi(B)TN+B\phi(TN)]\nonumber\\
&=&\phi(A)BTN+A\phi(B)TN+AB\phi(TN).\nonumber
\end{eqnarray}
The last two equations give us that
$[\phi(AB)-\phi(A)B-A\phi(B)]TN=0$ for all $T \in \mathbf{B(H)}$.
Thus we have
\begin{eqnarray}
\phi(AB)=\phi(A)B+A\phi(B).\nonumber
\end{eqnarray}

\par
 When $\Omega\neq\Phi$, for each $i \notin \Omega$, let
$Q_{i}=I_{-}$ be the projection in $\mathcal{L}_{i}$, define $M =
\Pi_{i\notin\Omega}Q_{i}^{\perp}$ (If $\Omega=\{1, 2,..., n\}$, we
take $M=I$). For each $i \in\Omega$, there exists an increasing
sequence $P_{i, k}$ of projections in $\mathcal{L}_{i}
\setminus\{I\}$ which strongly converges to $I$. Let
$E_{k}=\Pi_{i\in\Omega}P_{i, k}$ and $F_{k}=\Pi_{i\in\Omega}P_{i,
k}^{\bot}$. Then $\lim_{k\rightarrow\infty}E_{k}=I$. It is clear
that $E_{k}B(H)MF_{K}\subset\mathcal{R}(\mathcal{L})$ for all $k
\in\mathbb{N}$. Let $A, B \in Alg\mathcal{L}$, it follows from Lemma
\ref{abbas} that
\begin{eqnarray}
\phi(ABE_{k}TMF_{k})=\phi(AB)E_{k}TMF_{k}+AB\phi(E_{k}TMF_{k})\nonumber
\end{eqnarray}
for all $T\in \mathbf{B(H)}$ and $k \in \mathbb{N}$. On the other
hand, by Lemma \ref{abbas} we have
\begin{eqnarray}
\phi(ABE_{k}TMF_{k})&=&\phi(A)BE_{k}TMF_{k}+A\phi(BE_{k}TMF_{k})\nonumber\\
&=&\phi(A)BE_{k}TMF_{k}+A[\phi(B)E_{k}TMF_{k}+B\phi(E_{k}TMF_{k})]\nonumber\\
&=&\phi(A)BE_{k}TMF_{k}+A\phi(B)E_{k}TMF_{k}+AB\phi(E_{k}TMF_{k}).\nonumber
\end{eqnarray}
From the last two equations we have
$[\phi(AB)-\phi(A)B-A\phi(B)]E_{k}TMF_{k}=0$ for all $T\in
\mathbf{B(H)}$ and $k \in \mathbb{N}$. By independence of the nests
$\mathcal{L}_{i}, MF_{k}\neq 0$ for all $k\in \mathbb{N}$. Hence
\begin{eqnarray}
[\phi(AB)-\phi(A)B-A\phi(B)]E_{k}=0\nonumber
\end{eqnarray}
for all $k \in\mathbb{N}$. Letting $k\rightarrow\infty$, we have
that $\phi(AB)=\phi(A)B+A\phi(B)$. Hence $\phi$ is a derivation,
namely $0$ is a derivable point of $alg\mathcal{L}$ for norm
continuous linear mapping. This completes the proof.
\end{proof}

\end{document}